\documentclass[10pt]{amsart}
\usepackage{palatino}
\usepackage{graphicx}
\usepackage{geometry}
\newtheorem{proposition}{Proposition}
\newtheorem{theorem}{Theorem}
\newtheorem{lemma}{Lemma}
\newtheorem{corollary}{Corollary}

\theoremstyle{definition}
\newtheorem{definition}{Definition}

\newtheorem{remark}{Remark}

\newcommand{\xclass}[1]{\langle #1 \rangle}

\newcommand{\xnorm}[1]{ \Vert #1 \Vert }

\newcommand{\zz}[1]{\mathbb #1}

\title[Self-Intersections of
Random  Geodesics]{Self-Intersections of
Random  Geodesics on  Negatively Curved Surfaces}
\author{Steven P. Lalley} \address{University of Chicago\\ Department
of Statistics \\ 5734
University Avenue \\
Chicago IL 60637}
\email{lalley@galton.uchicago.edu}
\date{\today}
\thanks{Supported by NSF grant DMS-0405102}

\begin{document}

\begin{abstract}
We study the fluctuations of self-intersection counts of random
geodesic segments of length $t$ on a compact, negatively curved
surface in the limit of large $t$. If the initial direction vector of
the geodesic is chosen according to the \emph{Liouville measure}, then
it is not difficult to show that the number $N (t)$ of
self-intersections by time $t$ grows like $\kappa t^{2}$, where
$\kappa =\kappa_{M}$ is a positive constant depending on the surface
$M$.  We show that (for a smooth modification of $N (t)$) the
fluctuations are of size $t$, and the limit distribution is a weak
limit of Gaussian quadratic forms. We also show that the fluctuations
of \emph{localized}  self-intersection counts (that is, only
self-intersections in a fixed subset of $M$ are counted)  are
typically of size $t^{3/2}$, and the limit distribution is Gaussian.
\end{abstract}

\maketitle

\section{Fluctuations of Self-intersection Counts}\label{sec:introduction}

Choose a point $x$ and a direction $\theta$ at random on a compact,
negatively curved surface $M$, and let $\gamma (t)=\gamma
(t;x,\theta)$ be the (unit speed) geodesic ray in direction $\theta$
started at $x$.  For large $t$ the number $N (t)=N (\gamma [0,t])$ of
transversal\footnote{If the initial point $x$ and direction $\theta$
are chosen randomly (according to the normalized Liouville measure on
the unit tangent bundle) then there is probability $0$ that the
resulting geodesic will be periodic, so with probability $1$ every
self-intersection will necessarily be transversal.}
self-intersections of the geodesic segment $\gamma [0,t]$ will be of
order $t^{2}$; in fact, if $\kappa_{M}= (2\pi |M|)^{-1}$, where $|M|$
denotes the surface area of $M$, then
\begin{equation}\label{eq:lln}
	\lim_{t \rightarrow \infty} N (t)/t^{2} = \kappa_{M}/2
\end{equation}
with probability $1$. See section~\ref{sec:slln} below for the (easy)
proof. Furthermore, the empirical distribution of the self-intersection
points  converges to the uniform distribution on the
surface.  Similar results hold for a randomly chosen \emph{closed}
geodesic \cite{lalley:si1}: If from among all closed geodesics of
length $\leq L$ one is chosen at random, then the number of
self-intersections, normalized by $L^{2}$, will, with probability
approaching one as $L \rightarrow \infty$, be close to $\kappa_{g}$.
These results  have been extended \cite{pollicott-sharp} to the number
and distribution of self-intersections at angles in fixed intervals
$[\alpha ,\beta]$. Closed geodesics with \emph{no} self-intersections
have long been of interest in geometry --- see, for instance,
\cite{birman-series:1,birman-series:2} --- and recently, M. Mirzakhani
\cite{mirzakhani} found the asymptotic growth rate of the number of
simple closed geodesics of length $\leq t$ as $t \rightarrow
\infty$. That this number is not $0$ shows (in view of the Law of
Large Numbers \eqref{eq:lln} ) that there is substantial
variation in the random variable $N (t)$.

The primary objective of this paper is to investigate the
\emph{fluctuations} (second-order asymptotics) of the
self-intersection numbers.  One's first guess might be that these are
of order $t$, and this is indeed the case; however, lest this seem too
obvious we add that if one counts only self-intersections in a (nice)
sub-domain $U\subset M$ then the fluctuations are no longer of order
$t$, but rather $t^{3/2}$. (In fact we will only prove these
statements for smoothed versions of the counts.)  One might also guess
that the rescaled random variable $(N (t)-\kappa_{g}t^{2})/t$ should
converge in distribution as $t \rightarrow \infty$ to a Gaussian
distribution, but this, as we will show, is probably false: the limit
distribution is a weak limit of \emph{Gaussian quadratic forms}. (This
does not preclude the possibility that it is Gaussian, but the
arguments below will make it clear that this is unlikely.)  Weak
limits of Gaussian quadratic forms are known to occur in connection
with stationary processes exhibiting long-range dependence
\cite{taqqu}, \cite{major} (where they are known as \emph{Rosenblatt
distributions}), and also in the connection with certain types of
$U-$statistics \cite{serfling}.

\begin{definition}\label{definition:rosenblatt}
A \emph{Gaussian quadratic form} is a random variable (or its
distribution) of the form $\sum_{j=1}^{m} \theta_{j}Z_{j}^{2}$, where
the random variables $Z_{j}$ are independent, unit Gaussians and
$\theta_{j}$ are real scalars.
\end{definition}

Unfortunately, the study of fluctuations in the self-intersection
numbers $N (t)$ is complicated by the (infrequent) occurrence of
self-intersections at very small angles. We have not yet been able to
successfully resolve the technical issues created by such
self-intersections, and so we will state our main
result not for the counts $N (t)$ but rather for a
smoothed version $N_{\varphi} (t)$ defined as follows. Let $\varphi :
\zz{R} \rightarrow [0,\infty )$ be an \emph{even}, $C^{\infty}$,
nonnegative, $2\pi -$periodic function. For each self-intersection $i$
of the geodesic segment $\gamma [0,t]$, denote by $\theta_{i}\in
[0,2\pi )$ the angle of the self-intersection: in particular, if the
self-intersection occurs at times $0\leq s_{i}<t_{i}\leq t$, then
$\theta_{i}$ is the angle between the tangent vector to $\gamma$ at
$s_{i}$ and the tangent vector at $t_{i}$. Define the \emph{smoothed
self-intersection number}
\begin{equation}\label{eq:Nphi}
	N_{\varphi} (t)=N_{\varphi} (\gamma [0,t])=
	\sum_{i=1}^{N (t)} \varphi (\theta_{i}).
\end{equation}
Clearly, if $\varphi \equiv 1$ then $N_{\varphi} (t)=N (t)$.  We will
prove in section~\ref{sec:slln} that the smoothed self-intersection
numbers satisfy a strong law of large numbers (SLLN) analogous to
\eqref{eq:lln}:
\begin{equation}\label{eq:llnPhi}
	\lim_{t \rightarrow \infty} N_{\varphi}(t) /t^{2}=\kappa_{\varphi}/2
\end{equation}
where
\begin{equation}\label{eq:Kphi}
	\kappa_{\varphi}=\frac{1}{2\pi |M|}\int_{0}^{2\pi} \varphi
	(\theta) |\sin \theta | \,d\theta.
\end{equation}

\begin{theorem}\label{theorem:random}
Assume that the smoothing function $\varphi$ is zero in a neighborhood
$[-\alpha ,\alpha]$ of $0$ and positive in $(\alpha ,\pi -\alpha
)$. If the initial point $x$ and direction $\theta$ are chosen
randomly according to the Liouville measure, then as $t \rightarrow
\infty$,
\begin{equation}\label{eq:random}
	 \frac{N_{\varphi} (t)-\kappa_{\varphi}t^{2}}{t} \Longrightarrow \Psi 
\end{equation}
where $\Rightarrow$ indicates weak convergence to a distribution $\Psi
$ in the weak closure of the set of Gaussian quadratic forms. This
limiting distribution may depend on both the surface $M$ and the
smoothing function $\varphi$.
\end{theorem}

We have not been able to prove that the limit distribution $\Psi$ is
nondegenerate, nor that it is non-Gaussian, but this seems unlikely
(see Section~\ref{ssec:heuristics} below).  
It is naturally of interest to consider also fluctuations in the
empirical distribution of self-intersection points. Let $f:M
\rightarrow \zz{R}$ be any continuous function on the surface $M$. For
each self-intersection $i$ of the geodesic segment $\gamma [0,t]$,
denote by $\theta_{i}\in [-\pi ,\pi]$ the angle and $x_{i}\in M$ the
location of the self-intersection. Define the $f-$localized
self-intersection counts
\begin{equation}\label{eq:f-local}
	N_{\varphi ;f} (t) =N_{\varphi ;f} (\gamma [0,t])=
	\sum_{i=1}^{N (t)} f (x_{i})\varphi (\theta_{i}). 
\end{equation}
As for the global self-intersection counts, the localized
self-intersection random variables $N_{\varphi ;f} (t)$ obey a strong
law of large numbers: For $\nu_{L}-$almost every initial direction
vector, 
\begin{equation}\label{eq:lln-local}
	\lim_{t \rightarrow \infty} N_{\varphi ;f}
	(t)/t^{2}=\frac{\kappa_{\varphi}}{2} \int_{M}f (x) \,dx:=A_{\varphi;f} ,
\end{equation}
where $dx$ indicates the normalized surface area measure on the
surface $M$. See section~\ref{ssec:local-kernels} for the proof. What  is
interesting  about the localized self-intersection counts is that
their fluctuations are of a different order of magnitude than those of
the global variables $N_{\varphi} (t)$, at least for  functions
$f$ of small support:

\begin{theorem}\label{theorem:local}
For any compact, negatively curved surface $M$ there is a constant
$\varepsilon >0$ such that the following is true: If the smoothing
function $\varphi$ satisfies the hypotheses of
Theorem~\ref{theorem:random}; if $f:M \rightarrow \zz{R}$ is any
$C^{\infty}$, nonnegative function that is not identically $0$ and
whose support has diameter $<\varepsilon $; and if the initial
direction vector $\gamma (0)$ is chosen randomly according to the
Liouville measure, then for some $\sigma =\sigma_{\varphi ;f}>0$,
\begin{equation}\label{eq:clt-local}
	\frac{N_{\varphi ;f} (t)-A_{\varphi;f}  t^{2}}{\sigma t^{3/2}} \Longrightarrow \Phi 	
\end{equation}
as $t \rightarrow \infty$, where $\Phi$ is the standard (mean $0$,
variance $1$) Gaussian distribution.
\end{theorem}

Simple lower bounds for $\varepsilon >0$ will be given in
Section~\ref{sec:local}, where Theorem~\ref{theorem:local} will be
proved.  Theorem~\ref{theorem:random} will be proved in
section~\ref{sec:weakConvergence}, using an extension of a mixing
result of Dolgopyat \cite{dolgopyat}. The proof relies on a
representation of the self-intersection counts in terms of what we dub
the \emph{intersection kernels} for the geodesic flow; these are
studied in section~\ref{sec:kernels}. For completeness, we present
proofs of the strong laws of large numbers \eqref{eq:lln} and
\eqref{eq:llnPhi} in section~\ref{sec:slln}, and of
\eqref{eq:lln-local} in section~\ref{ssec:local-kernels} .

\section{Intersection Kernel}\label{sec:kernels}

\subsection{The intersection kernel}\label{ssec:i-kernel} Geodesics on
any surface\footnote{Here and throughout the paper the term
\emph{geodesic} will be used either to indicate a geodesic path in the
surface $M$ or its lift to the unit tangent bundle $SM$; the meaning
should be clear from context. Two geodesic segments will be said to
intersect if their projections to the surface $M$ intersect.},
regardless of its curvature, look locally like straight
lines. Consequently, for any \emph{compact} surface $M$ with smooth
Riemannian metric there exists $\varrho >0$ such that if $\alpha$ and
$\beta $ are geodesic segments of length $\leq \varrho$ then $\alpha$
and $\beta$ intersect transversally, if at all, in at most one
point. Thus, if $\gamma$ is not periodic then the smoothed
self-intersection number $N_{\varphi } (L)=N_{\varphi }(\gamma [0,L])$
can be computed by partitioning $\gamma[0,L]$ into nonoverlapping
segments of common length $\delta \leq \varrho$ and counting the
number of pairs that intersect transversally.  Let $\alpha_{i}$ and
$\alpha_{j}$ be two such segments; then the event that these segments
intersect is completely determined by their initial points and
directions, as is the angle of intersection.

\begin{definition}\label{definition:intersectionKernel}
The \emph{intersection kernel} $H_{\delta}:SM\times SM \rightarrow
\zz{R}_{+}$ is the nonnegative function that takes the value
$H_{\delta} (u,v)=\varphi (\theta)$ if the geodesic segments of length
$\delta$ with initial tangent vectors $u$ and $v$ intersect
transversally at angle $\theta$, and $H_{\delta} (u,v)=0$ otherwise.
\end{definition}

The dependence on the smoothing function $\varphi$ is suppressed, as
$\varphi$ will be fixed throughout the paper.  Because $\varphi$ is
assumed to be even and $\pi -$periodic, $H_{\delta}$ is
symmetric in its arguments $u,v$. The intersection kernel 
determines the smoothed self-intersection numbers as follows: If
$L=n\delta $ is an integer multiple of $\delta$, then for any geodesic
$\gamma $,
\begin{equation}\label{eq:u-statistic}
	N_{\varphi} (L)= N_{\varphi} (\gamma[0,L])=\frac{1}{2} \sum_{i=1}^{n} \sum_{j=1}^{n}
	H_{\delta} (\gamma (i\delta ), \gamma (j\delta)).
\end{equation}
The factor of $1/2$ compensates for the double-counting that results
from letting both indices of summation $i,j$ range over all $n$
geodesic segments.  Note that the diagonal terms $H_{\delta} (\gamma
(i\delta ),\gamma (i\delta ))$ in this sum are all $0$, because the
segment $\gamma (i\delta )$ does not intersect itself
{transversally}.

\subsection{The associated integral
operators}\label{ssec:integralOperator}
The intersection kernel  $H_{\delta } (u,v)$ is symmetric in its
arguments and Borel measurable, but not continuous, because 
self-intersections can be created or destroyed by small perturbations
of the initial vectors $u,v$.  Nevertheless,  $H_{\delta}$ induces a
self-adjoint integral operator on the Hilbert space $L^{2} (\nu_{L})$ by
\begin{equation}\label{eq:kernel}
	H_{\delta} \psi (u)=\int_{v\in SM} H_{\delta} (u,v) \psi (v)
	\,d\nu_{L} (v). 
\end{equation}

\begin{lemma}\label{lemma:H1}
For all sufficiently small $\delta >0$,
\begin{equation}\label{eq:H1}
	H_{\delta} 1(u):
	 =\int H_{\delta} (u,v)\,d\nu_{L} (v) =
	 \delta^{2}\kappa_{\varphi}
\end{equation}
for all $u\in M$. Thus, the constant function $1$ is an eigenfunction
of the operator $H_{\delta}$, and consequently the normalized kernel
$H_{\delta} (u,v)/\delta^{2}\kappa_{\varphi}$ is a Markov kernel.
\end{lemma}

\begin{remark}\label{remark:Hd}
The integral $H_{\delta}1 (u)$ is the expectation of $\varphi
(\theta)$ where $\theta$ is the angle in which a randomly
chosen geodesic segment of length $\delta$  intersects the geodesic
segment of length $\delta$ with initial tangent vector $u$. The
assertion of the lemma is that this expectation does not depend on the
initial tangent vector $u$, even for surfaces $M$ of variable negative
curvature. 
\end{remark}

\begin{proof}
Denote by $\alpha =\gamma ( [0,\delta];u)$ the geodesic segment of
length $\delta$ with initial tangent vector $u$.  For small $\delta
>0$ and fixed angle $\theta$, the set of points $x\in S$ such that a
geodesic segment of length $\delta$ with initial base point $x$
intersects $\alpha $ at angle $\theta $ is approximately a
rhombus of side $\delta$ with an interior angle $\theta$. The area of
such a rhombus is $\delta^{2}|\sin \theta |$. Hence, as $\delta \rightarrow 0$,
\begin{equation}\label{eq:sim}
	\int H_{\delta} (u,v)\,d\nu_{L} (v) \sim
	\delta^{2}\int_{0}^{2\pi} \varphi (\theta) |\sin \theta | \,d\theta / (
	2\pi |M|)=\delta^{2} \kappa_{\varphi},
\end{equation}
and the relation $\sim$ holds uniformly for $u\in SM$.

It remains to show that the approximate equality $\sim$ is actually an
equality for small $\delta >0$, equivalently, that the value of the
integral stabilizes as $\delta \rightarrow 0$. Assume that $\delta >0$
is sufficiently small that any two distinct geodesic segments of
length $\delta$ intersect transversally at most once. Consider the
geodesic segments of length $\delta$ with initial direction vectors
$u$ and $v$. For any  integer $m\geq 2$, each of these segments
can be partitioned into $m$ nonoverlapping sub-segments (each open on
one end and closed on the other) of length $\delta /m$. At most one
pair of these constituent sub-segments can intersect; hence,
\[
	 H_{\delta} (u,v) =\sum_{i=0}^{m-1} \sum_{j=0}^{m-1}
	 	    H_{\delta /m} (\gamma (i\delta ;u),\gamma (j\delta
		    ;v)) .	
\]
Integrating over $v$ with respect to the Liouville measure $\nu_{L}$
and using the invariance of $\nu_{L}$ relative to the geodesic flow
we obtain that
\[
	H_{\delta}1 (u)
	= \sum_{i=0}^{m-1} mH_{\delta /m}1 (\gamma (i\delta ;u)).
\]
Let $m \rightarrow \infty$ and use the approximation \eqref{eq:sim}
(with $\delta$ replaced by $\delta /m$); since this approximation
holds uniformly, it follows that $H_{\delta}1 (u)=\delta^{2}\kappa_{S}$.
\end{proof}

\begin{lemma}\label{lemma:compactness}
For each $\delta >0$ sufficiently small, the integral operator
$H_{\delta}$ on $L^{2} (\nu_{L})$ is compact.
\end{lemma}

\begin{proof}
If the kernel $H_{\delta} (u,v)$ were jointly continuous in its
arguments $u,v$ then this would follow by standard results about
integral operators --- see, e.g., \cite{widom}. Since $H_{\delta}$ is
not continuous, these standard results do not apply; nevertheless, the
argument for compactness is elementary.  It suffices to show that the
mapping $u \mapsto H_{\delta} (u,\cdot)$ is continuous relative to the
$L^{2}-$norm. Take $u,u' \in SM$, and let $\alpha ,\alpha '$ be the
geodesic segments of length $\delta$ started at $u,u'$, respectively.
If $u,u'$ are close, then the geodesic segments $\alpha ,\alpha '$ are
also close. Hence, for all but very small angles $\theta$ the set of
points $x\in M$ such that a geodesic segment of length $\delta$ with
initial base point $x$ intersects $\alpha$ at angle $\theta $ but does
not intersect $\alpha '$ is small. Consequently, the functions
$H_{\delta} (u,\cdot)$ and $H_{\delta} (u',\cdot)$ differ on a set of
small measure.
\end{proof}

Lemma~\ref{lemma:compactness} implies that the  Hilbert-Schmidt
theory applies. In particular, the non-zero spectrum of $H_{\delta}$
consists of isolated real eigenvalues $\lambda_{j}$ of finite
multiplicity (and listed according to multiplicity). The
corresponding (real) eigenfunctions $\psi_{j}$ can be chosen so as
to consititute an orthonormal basis of $L^{2} (\nu_{L})$, and the
eigenvalue sequence $\lambda_{j}$ is square-summable.

\begin{lemma}\label{lemma:mixing}
The Markov kernel $\bar{H}_{\delta}:=H_{\delta}/\delta^{2}\kappa_{\varphi}$
satisfies the \emph{Doeblin condition}: There exist an integer $n\geq
1$ and a positive real number $\varepsilon$ such that
\begin{equation}\label{eq:doeblin}
	\bar{H}^{n}_{\delta} (u,v)\geq \varepsilon  \quad \text{for
	all} \;\; u,v\in SM.
\end{equation}
\end{lemma}

\begin{proof}
Chose $n$ so large that for any two points $x,y\in S$ there is a
sequence $\{x_{i} \}_{0\leq i\leq n}$ of $n+1$ points beginning with
$x_{0}=x$ and ending at $x_{n}=y$, and such that each successive pair
$x_{i},x_{i+1}$ are at distance $<\delta /4$. Then for any two
geodesic segments $\alpha ,\beta$ of length $\delta$ on $S$ there is a
chain of $n+1$ geodesic segments $\alpha_{i}$, all of length $\delta$,
beginning at $\alpha_{0}=\alpha$ and ending at $\alpha_{n}=\beta$,
such that any two successive segments $\alpha_{i}$ and $\alpha_{i+1}$
intersect transversally.  Since the intersections are transversal, the
initial points and directions of these segments can be jiggled
slightly without undoing any of the transversal intersections. This
implies \eqref{eq:doeblin}.
\end{proof}

\begin{corollary}\label{corollary:doeblin}
The eigenvalue $\delta^{2}\kappa_{\varphi}$ is a simple eigenvalue of the
integral operator $H_{\delta}$, and the rest of the spectrum lies in
a disk of radius $<\delta^{2}\kappa_{\varphi}$.
\end{corollary}

\begin{proof}
This is a standard result in the theory of Markov operators.
\end{proof}

\begin{corollary}\label{corollary:meanZeroEigenfunctions}
For every $j\geq 2$ the eigenfunction $\psi_{j}$ has mean zero
relative to $\nu_{L}$, and distinct eigenfunctions are uncorrelated.
\end{corollary}

\begin{proof}
The spectral theorem guarantees   orthogonality of the
eigenfunctions. The key point is that $\psi_{1}=1$ is an
eigenfunction, and so the orthogonality $\psi_{j}\perp \psi_{1}$
implies that each $\psi_{j}$ for $j\geq 2$ has mean zero.
\end{proof}

\begin{lemma}\label{lemma:nontrivialEigenvalues}
If $\delta >0$ is sufficiently small then $H_{\delta}$ has
eigenvalues other than $0$ and $\lambda_{1} (\delta)$.
\end{lemma}

\begin{proof}
Otherwise, the Markov operator  $\bar{H}_{\delta}$ would be a
projection operator: for every $\psi \in L^{2} (\nu_{L})$ the
function $\bar{H}_{\delta}\psi$ would be constant. But if $\delta
>0$ is small, this is obviously not the  case.
\end{proof}

\subsection{Smoothing}\label{ssec:smoothing} The discontinuity of the
kernel $H_{\delta}$ creates certain technical problems: for instance,
the eigenfunctions $\psi_{j}$ need not be continuous. Thus, it will be
to our advantage to approximate $H_{\delta}$ by a smooth kernel
$K_{\delta}$ in such a way that the sums \eqref{eq:u-statistic} are
not too badly disturbed when $H_{\delta}$ is replaced by the
approximation $K_{\delta}$. It is \emph{solely} for the purpose of
constructing this approximation that the restrictions in
Theorem~\ref{theorem:random} on the smoothing kernel $\varphi$ --- in
particular, that it vanishes in a neighborhood of $0$ --- are needed.

Let $p( s)$ be an even, $C^{\infty}$ probability density on $\zz{R}$
with support contained in the interval $[-1,1]$ Define $\mathcal{D}$
to be the set of pairs $(u,v)\in SM\times SM$ such that the geodesics
through $u$ and $v$ cross, if at all, \emph{transversally}; this set
is dense in $SM\times SM$.  For any pair $(u,v)\in \mathcal{D}$, the
(two-sided) geodesics $\{\gamma (s;u) \}_{s\geq 0}$ and $\{\gamma
(t;v) \}_{t\geq 0}$ will intersect in at most countably many points,
which can be labeled $(s_{i},t_{i})$ where the entries $s_{i}$ and
$t_{i}$ denote the signed distances along the two geodesics from their
origins $u,v$ where the intersection occurs. (The
intersections do not generally occur in the same (time) order along
the two geodesics.) Let $\theta_{i}\in
(0,\pi)$ be the angle of crossing at the $i$th intersection, and
define
\begin{equation}\label{eq:kernelK}
	K_{\delta} (u,v) =\sum_{i}  \delta^{-2}p (s_{i}/\delta ) p
	(t_{i}/\delta ) \varphi (\theta_{i}).
\end{equation}

\begin{lemma}\label{lemma:well-def}
Assume that the smoothing function $\varphi$ is zero in a neighborhood
$[-\alpha ,\alpha]$ of $0$.
If $\delta >0$ is sufficiently small then for any pair $(u,v)\in
\mathcal{D}$, the sum \eqref{eq:kernelK} contains at most
one nonzero term,  so $K_{\delta} (u,v)$ is well-defined and
finite. Furthermore, the function $K_{\delta} (u,v)$ extends to a
$C^{\infty}$, symmetric function on $SM\times SM$, by setting 
$K_{\delta} (u,v)=0$ for all $(u,v)\not \in \mathcal{D}$.
\end{lemma}

\begin{proof}
Since $\varphi$ is even, the kernel $K_{\delta}$ is symmetric if the
sum \eqref{eq:kernelK} is finite. Now in order that the $i$th term of
the sum \eqref{eq:kernelK} be nonzero, the distances $|s_{i}|$ and
$|t_{i}|$ must both be smaller than $\delta $, because $p$ has support
contained in $[-1,1]$. But if $\delta >0$ is sufficiently small then
any two geodesic segments of length $2\delta$ will intersect
transversally at most once.  Thus, the sum \eqref{eq:kernelK} has at
most one nonzero term.

Since geodesics vary smoothly with their initial conditions, the
intersection distances $s_{i},t_{i}$ and angles $\theta_{i}$ vary
smoothly with $u,v$  in $\mathcal{D}$.  Consequently, the
function $K_{\delta}$ is $C^{\infty}$ in $\mathcal{D}$. But
as $(u,v)$ approaches the boundary of $\mathcal{D}$, the angle(s) of
intersection of the geodesics through $u$ and $v$ must approach
zero. Hence, by the assumption on $\varphi$, the kernel $K_{\delta}$
vanishes near $\partial \mathcal{D}$.
\end{proof}

\begin{lemma}\label{lemma:siFormula}
For any geodesic $\gamma$, the smoothed self-intersection number
$N_{\varphi} (t)$ of the segment $\gamma [0,t]$ satisfies the
inequalities
\begin{equation}\label{eq:bounds}
	\frac{1}{2} \iint_{[\delta ,t-\delta ]^{2}} K_{\delta}
	(\gamma (r),\gamma (s)) \, dr\,ds 
	\leq N_{\varphi} (t) \leq 
	\frac{1}{2} \iint_{[-\delta ,t+\delta ]^{2}} K_{\delta}
	(\gamma (r),\gamma (s)) \, dr\,ds .
\end{equation}
\end{lemma}

\begin{proof}
Suppose that  $\gamma$ has a self-intersection at some $(r,s)$, that
is, the vectors $\gamma (r)$ and $\gamma (s)$ lie over the same base
point in $M$. Then by definition of $K_{\delta}$, since $h$ is a
probability density supported by $[-\delta /2,\delta /2]$, the
integral of $K_{\delta} (\gamma (r'),\gamma (s'))$ over the square of
side $\delta$ centered at $(r,s)$ must be $1$. Hence, both bounding
integrals in \eqref{eq:bounds} count each such self-intersection
$(r,s)$ with weight $1$. The only self-intersections not counted
correctly are those $(r,s)$ where either $r$ or $s$ lies within
$\delta$ of one of the time endpoints $0$ or $t$. Adjusting the limits of
integration by $\pm \delta$ compensates for these boundary errors.
\end{proof}

\begin{remark}\label{remark:errors}
The inequalities \eqref{eq:bounds} imply that the double integral on
the right side of \eqref{eq:bounds} is bounded above by $N_{\varphi} (\gamma
[-\delta ,t+\delta])$, and the  double integral on the left is bounded
below by $N_{\varphi} (\gamma [\delta ,t-\delta])$. Therefore, the errors in the
inequalities \eqref{eq:bounds} are no larger than
\begin{equation}\label{eq:discrepancy}
	N_{\varphi}(\gamma [-\delta ,t+\delta])-N_{\varphi}(\gamma
	 [\delta,t-\delta]).
\end{equation}
\qed
\end{remark}

The smoothed kernels $K_{\delta} (u,v)$ enjoy all of the properties
enumerated for the intersection kernels $H_{\delta} (u,v)$ in
sec.~\ref{ssec:integralOperator} above, provided the smoothing window
$\delta >0$ is sufficiently small. In particular,

\begin{itemize}
\item [(P1)] The constant function $1$ is an eigenfunction of
$K_{\delta}$, with eigenvalue $\kappa_{\varphi}$.
\item [(P2)] The integral operator $K_{\delta}$ is compact.
\item [(P3)] The Markov kernel $K_{\delta}/\kappa_{\varphi}$ satisfies
the Doeblin condition. 
\item [(P4)] The eigenvalue $\kappa_{\varphi}$ is simple, and the rest
of the spectrum lies in some
$[-\kappa_{\varphi}+\varepsilon,\kappa_{\varphi}-\varepsilon]$. 
\item [(P5)] Distinct eigenfunctions are uncorrelated, and except for
the constant eigenfunction have mean $0$.
\end{itemize} 

These may all be proved by mimicking the proofs of the analogous
assertions for the self-intersection kernels $H_{\delta}$. In the
special case that the smoothing density $p$ in the definition
\eqref{eq:kernelK} is not only even but also nondecreasing on
$(-\infty ,0]$, property (P1) can be deduced directly from Lemma
\ref{lemma:H1}, because in this case the kernel $K_{\delta}$ is a
convex combination of the kernels $H_{\varepsilon}/\varepsilon^{2}$,
where $0<\varepsilon \leq \delta$. Only Property (P1) will be needed
for the proof of Theorem~\ref{theorem:random}, and we will have no
need to consider smoothing densities $p$ that are not monotone on
$(-\infty ,0]$, so we refrain from spelling out the details of the
proofs of (P1)--- (P5) in the general case.

\section{SLLN for Self-Intersections}\label{sec:slln}

\subsection{SLLN for the smoothed self-intersection
numbers}\label{ssec:llnSmooth} According to
Lemmas~\ref{lemma:well-def}--\ref{lemma:siFormula}, if the smoothing
function $\varphi$ is zero in a neighborhood $[-\alpha,\alpha]$ of
$0$, then the intersection kernel $H_{\delta}$ can be approximated by
a continuous kernel $K_{\delta}$ in such a way that the smoothed
self-intersection number  $N_{\varphi} (t)$ is well-approximated by 
\[
	N^{*} (\gamma [0,t]):=\frac{1}{2} \iint_{[0,t]} K_{\delta} (\gamma (r),\gamma (s))
	\,dr ds.
\]
For the strong law of large numbers \eqref{eq:llnPhi}, only a crude
bound on the error in this approximation is needed.  By
Remark~\ref{remark:errors}, the error is bounded by
\eqref{eq:discrepancy}; and by another use of the double inequalities
\eqref{eq:bounds}, it follows that the discrepancy is bounded by the
difference $N^{*} (\gamma [-\delta ,t+\delta])-N^{*} (\gamma [\delta
,t-\delta])$.  Consequently, since $\delta >0$ can be chosen
arbitrarily small, to prove the strong law of large numbers
\eqref{eq:llnPhi} it suffices to show that for $\delta >0$
sufficiently small and for $\nu_{L}-$almost every initial point
$\gamma (0)$,
\begin{equation}\label{eq:LLNobj}
	\lim_{t \rightarrow \infty} t^{-2}\iint_{[0,t]^{2}} K_{\delta}
	(\gamma (s_{1}),\gamma (s_{2})) \,ds_{1} ds_{2} =\kappa_{\varphi }
\end{equation}
The following proposition implies that for any probability
measure $\mu$ on $SM$ that is invariant and ergodic under the geodesic
flow, for $\mu -$almost every initial point $\gamma (0)$,
\begin{equation}\label{eq:llnGen}
	\lim_{t \rightarrow \infty} t^{-2}\iint_{[0,t]^{2}} K_{\delta}
	(\gamma (s_{1}),\gamma (s_{2})) \,ds_{1} ds_{2} =
	\iint_{SM\times SM} K_{\delta} (x,y) \, d\mu (x)d\mu (y).
\end{equation}
That the expectation on the right equals $\kappa_{\varphi}$ for $\mu
=\nu_{L}$ follows from  property (P1) above.

\begin{proposition}\label{proposition:ergodic}
Let $(\mathcal{X},d)$ be a compact metric space and let
$K:\mathcal{X}^{2} \rightarrow \zz{R}$ be continuous. If $\mu$ is a
Borel probability measure on $\mathcal{X}$ and $T:\mathcal{X}
\rightarrow \mathcal{X}$ is an ergodic, measure-preserving
transformation (not necessarily continuous) relative to $\mu$, then
\begin{equation}\label{eq:ergodic}
	\lim_{n \rightarrow \infty}
	\frac{1}{n^{2}}\sum_{i=1}^{n}\sum_{j=1}^{n} K (T^{i}x,T^{j}x)
	= \iint_{\mathcal{X}\times \mathcal{X}} K (y,z) \,d\mu (y) d\mu (z)
\end{equation}
for $\mu -$almost every $x$.
\end{proposition}

\begin{proof}
The function $K$ is bounded, since it is continuous,
so the double integral in \eqref{eq:ergodic} is well-defined and
finite. Furthermore, the set of functions $K_{x}$ defined by $K_{x}
(y) :=K (x,y)$, where $x$ ranges over the space $\mathcal{X}$, is
equicontinuous, and the function
\[
	\bar{K}_{x}:=\int_{\mathcal{X}}K_{x} (y) \,d\mu (y)
\]
is continuous in $x$.  The equicontinuity of the functions $K_{x}$
implies, by the Arzela-Ascoli theorem, that for any $\varepsilon >0$
there is a finite subset $F_{\varepsilon}=\{ x_{i}\}_{1\leq i\leq I}$
such that for any $x\in SM$ there is at least one index $i\leq I$ such
that
\begin{equation*}
	\xnorm{K_{x}-K_{x_{i}}}_{\infty}<\varepsilon .
\end{equation*}
It follows that the time average of $K_{x}$ along any trajectory
differs from the corresponding time average of $K_{x_{i}}$ by less
than $\varepsilon$. But since the set $F_{\varepsilon}$ is finite,
Birkhoff's ergodic theorem implies that if the initial point and
direction of $\gamma$ are chosen randomly according to $\nu_{L}$ then
with probability one, for each $x_{i}\in F_{\varepsilon}$,
\begin{equation*}
	\lim_{t \rightarrow \infty} \frac{1}{t} \int_{0}^{t}
	K (x_{i},\gamma (s)) \,ds = \int K (x_{i},
	y) \,d\nu_{L} (y).
\end{equation*}
It therefore follows from equicontinuity (let $\varepsilon \rightarrow
0$) and the continuity in $x$ of the averages $\bar{K}_{x}$ that
almost surely
\begin{equation*}
	\lim_{t \rightarrow \infty} \frac{1}{t} \int_{0}^{t}
	K (x,\gamma (s)) \,ds = \int K (x,
	y) \,d\nu_{L} (y) 
\end{equation*}
\emph{uniformly} for $x\in \mathcal{X}$. The uniformity of this convergence
guarantees that \eqref{eq:ergodic} holds
$\mu -$almost surely. 
\end{proof}

\begin{remark}\label{remark:subtlety}
Wiener's multi-parameter ergodic theorem (\cite{wiener}, Theorems I''
-- II'') implies that under much weaker hypotheses\footnote{Wiener
requires that $K\in L^{2} (\mu \times \mu)$. For still weaker
hypotheses, see \cite{krengel}.} on the function
$K$,
\[
	\lim_{n \rightarrow \infty}
	\frac{1}{n^{2}}\sum_{i=1}^{n}\sum_{j=1}^{n} K (T^{i}x,T^{j}y)
	= \iint_{\mathcal{X}\times \mathcal{X}} K (u,v) \,d\mu (u)
	d\mu (v)
	\quad \text{for} \; (\mu \times \mu )-\text{every} \;(x,y).
\]
Proposition~\ref{proposition:ergodic} cannot be deduced from this, as
the diagonal of $\mathcal{X}\times \mathcal{X}$ has $(\mu \times
\mu)-$measure $0$.  In fact, it is \emph{not} generally true that the convergence
\eqref{eq:ergodic} holds for functions $K (x,y)$ that are not
continuous, even if the measure-preserving transformation $T$ is
mixing.  A simple example can be constructed as follows. Let
$R=R_{\theta}$ be an irrational rotation of the circle $S^{1}$, and
let $\sigma :\Sigma \rightarrow \Sigma$ be the shift on the space of
all one-sided sequences $\omega =\omega_{1}\omega_{2}\dotsb$ with
entries $\pm 1$. For $x\in S^{1}$ and $\omega \in \Sigma$, define
\[
	T (x,\omega)= (R^{\omega_{1}}x,\sigma \omega );
\]
this is a mixing, measure-preserving transformation relative to $\lambda =$
Lebesgue$\times \mu$,  where $\mu $ is the product Bernoulli-$(1/2)$
measure on $\Sigma$. Now for $x,y\in S^{1}$ and $\omega ,\omega '\in
\Sigma$, define 
\begin{align*}
	K ((x,\omega),(y,\omega '))=K (x,y) &=1 \quad \text{if} \; y-x
	\in  \xclass{\theta}\\
	&=0 \quad \text{otherwise,}
\end{align*}
where $\xclass{\theta}$ denotes the (countable) subgroup of $S^{1}$
generated by $\theta$.  Clearly, $K=0$ almost surely relative to
the product measure $\lambda \times \lambda$, but $K
(T^{i}x,T^{j}x)=1$ along every orbit of $T$. Therefore,
\eqref{eq:ergodic} fails. It is not difficult to modify the
function $K$ so that the limit fails to exist with probability one.
\qed
\end{remark}

\subsection{SLLN for self-intersections}\label{ssec:sllnSI}

It is not much more difficult to prove the law of large numbers
\eqref{eq:lln}, by incorporating an additional observation from
\cite{lalley:si1}.  First, observe that
Proposition~\ref{proposition:ergodic} can be reformulated as a
statement about weak convergence of empirical distributions:

\begin{corollary}\label{corollary:empirical}
Let $(\mathcal{X},d)$ be a compact metric space, $\mu$  a
Borel probability measure on $\mathcal{X}$, and $T:\mathcal{X}
\rightarrow \mathcal{X}$  an ergodic, measure-preserving
transformation relative to $\mu$. For each $x\in \mathcal{X}$ and
$n\geq 1$, let 
\begin{equation}\label{eq:empirical}
	\mu^{x}_{n} :=\frac{1}{n^{2}}\sum_{i=1}^{n} \sum_{j=1}^{n}
	\delta_{(T^{i}x,T^{j}x)} 
\end{equation}
be the empirical measure that puts mass $1/n^{2}$ at each point
$(T^{i}x,T^{j}x)$. Then for $\mu -$almost every $x\in \mathcal{X}$,
\begin{equation}\label{eq:SLLNempirical}
	\mu^{x}_{n} \stackrel{w*}{\longrightarrow} \mu \times \mu . 
\end{equation}
Consequently, if $U\subset \mathcal{X}\times \mathcal{X}$ is any
Borel measurable set whose topological boundary satisfies $\mu \times
\mu (\partial U)=0$, then
 for $\mu -$almost every $x\in \mathcal{X}$,
\begin{equation}\label{eq:portmanteau}
	\lim_{n \rightarrow \infty}
	\frac{1}{n^{2}}\sum_{i=1}^{n}\sum_{j=1}^{n}
		\mathbf{1}_{U}(T^{i}x,T^{j}x) =\mu \times \mu (U).
\end{equation}
\end{corollary}

\begin{proof}
The assertion \eqref{eq:SLLNempirical} is just an equivalent form of
\eqref{eq:ergodic}, by the definition of weak-$*$ convergence. Given
\eqref{eq:SLLNempirical}, the statement \eqref{eq:portmanteau} follows
by elementary analysis (see, for instance, Billingsley's ``Portmanteau
Theorem'', \cite{billingsley}, Theorem 2.1).
\end{proof}

\begin{proof}
[Proof of SLLN \eqref{eq:lln}] Fix $\delta >0$ small, and consider the
representation \eqref{eq:u-statistic} for $N (L)$ when
$L=n\delta$. For $\varphi \equiv 1$, the intersection kernel
$H_{\delta} (x,y)$ takes the form of an indicator function $H_{\delta}
(x,y)=\mathbf{1}_{U} (x,y)$, where $U=U_{\delta }$ is the set of pairs
$(x,y)$ such that the geodesic segments of length $\delta$ with
initial points $x,y$ intersect transversally. It is easily checked
(see \cite{lalley:si1}) that the boundary of this set has
$\nu_{L}\times \nu_{L}-$measure $0$, so
Corollary~\ref{corollary:empirical} implies that for almost every
initial point,
\[
	\lim_{n \rightarrow \infty}\frac{1}{n^{2}} N (\gamma
	[0,n\delta])= \nu_{L}\times \nu_{L} (U_{\delta}).
\]
By the same argument as in the proof of Lemma~\ref{lemma:H1}, the
constant on the right is $\sim \kappa_{M}\delta^{2}$ as $\delta
\rightarrow 0$. Thus, the SLLN \eqref{eq:lln} follows.
\end{proof}

\section{Weak Convergence of Fluctuations}\label{sec:weakConvergence}
\subsection{Heuristics}\label{ssec:heuristics}
We begin by using the results of sections
\ref{ssec:i-kernel}---\ref{ssec:integralOperator} to give a compelling
--- but non-rigorous --- explanation of the weak convergence asserted
in Theorem~\ref{theorem:random}.  The Hilbert-Schmidt theorem asserts
that a symmetric integral kernel in the class $L^{2} (\nu_{L}\times
\nu_{L})$ has an $L^{2}-$convergent eigenfunction expansion. The
intersection kernel $H_{\delta} (u,v)$ meets the requirements of this
theorem, and so its eigenfunction expansion converges in $L^{2}
(\nu_{L}\times \nu_{L})$:
\begin{equation}\label{eq:eigenfunctionExp}
	H_{\delta} (u,v)=\sum_{k=1}^{\infty} \lambda_{k} 
	\psi_{k} (u) \psi_{k} (v).
\end{equation}
The $L^{2}-$convergence of the series does not, of course, imply
\emph{pointwise} convergence. Nevertheless, let's proceed formally,
ignoring convergence issues: Recall
(Corollary~\ref{corollary:meanZeroEigenfunctions}) that the
eigenfunctions are mutually uncorrelated, and so all except the constant
eigenfunction $\psi_{1}$ have mean zero relative to $\nu_{L}$. Thus,
the representation \eqref{eq:u-statistic} of the intersection number
$N_{\varphi} (n\delta)$ can be rewritten as follows:
\begin{align}\label{eq:heuristics}
	N_{\varphi } (n\delta) - (n\delta )^{2}\kappa_{g}&=  
	\frac{1}{2} \sum_{i=1}^{n} \sum_{j=1}^{n} H_{\delta} (\gamma
	(i\delta ), \gamma (j\delta)) - (n\delta )^{2}\kappa_{g}\\
\notag 	&=	\frac{1}{2} \sum_{i=1}^{n} \sum_{j=1}^{n}
	  \sum_{k=2}^{\infty} \lambda_{k} (\delta) \psi_{k} (\gamma
	  (i\delta ))\psi_{k} (\gamma (j\delta )) \\
\notag 	  &=\frac{1}{2} \sum_{k=2}^{\infty} \lambda_{k} (\delta) \left(
	  \sum_{i=1}^{n} \psi_{k} (\gamma 
	  (i\delta ))\right)^{2}.
\end{align}
If the eigenfunctions $\psi_{j}$ were H\"{o}lder
continuous, the central limit theorem for the geodesic flow
\cite{ratner:clt} would imply that  for
any finite $K$ the joint distribution of the random vector 
\begin{equation}\label{eq:K-vector}
	\left( \frac{1}{\sqrt{n}}\sum_{i=1}^{n} \psi_{k} (\gamma
	(i\delta ))\right)_{2\leq k\leq K}
\end{equation}
converges, as $n \rightarrow \infty$, to a (possibly degenerate)
$K-$variate Gaussian distribution centered at the origin. (The central
limit theorem in \cite{ratner:clt} is stated only for the case $K=1$,
but the general case follows by standard weak convergence arguments
[the ``Cramer-Wold device''], as in \cite{billingsley}, ch.~1.) Hence,
for every $K<\infty$ the distribution of the truncated sum
\begin{equation}\label{eq:K-form}
	\frac{1}{n} \sum_{k=2}^{K} \lambda_{k} (\delta) \left(
	\sum_{i=1}^{n} \psi_{k} (\gamma 
	  (i\delta ))\right)^{2}
\end{equation}
should converge, as $n \rightarrow \infty$, to that of a quadratic
form in the entries of the limiting Gaussian
distribution. \footnote{That the limit distribution has the same form
as required by Definition~\ref{definition:rosenblatt} follows by the
spectral theorem for symmetric matrices and elementary properties of
the multivariate Gaussian distribution, as follows. Suppose that the
limit distribution of the random vector \eqref{eq:K-vector} is
mean-zero Gaussian with (possibly degenerate) covariance matrix
$\Sigma$; this distribution is the same as that of $\Sigma^{1/2}Z$,
where $Z$ is a Gaussian random vector with mean zero and identity
covariance matrix. Let $\Lambda$ be the diagonal matrix with diagonal
entries $\lambda_{j} (\delta)$. Then the limit distribution of
\eqref{eq:K-form} is identical to that of $Z^{T}MZ$, where
$M=\Sigma^{1/2}\Lambda \Sigma^{1/2}$. But the matrix $M$ is symmetric,
so it may be factored as $M=U^{T} D U$, where $U$ is an orthogonal
matrix and $D$ is diagonal.  Now if $Z$ is mean-zero Gaussian with the
identity covariance matrix, then so is $UZ$, since $U$ is
orthogonal. Thus, $Z^{T}MZ$ has the same form as in
Definition~\ref{definition:rosenblatt}, where $\theta_{j}$ are the
diagonal entries of $D$. }

There are, obviously, two problems with this argument. First, the
central limit theorem requires that the functions $\psi_{j}$ be
H\"{o}lder continuous; but since the intersection kernels $H_{\delta}$
are not continuous, their eigenfunctions $\psi_{j}$ will not be
continuous either. This problem could be circumvented by
smoothing. But second, the convergence of the infinite series in
\eqref{eq:heuristics} and the interchange of limits requires
justification. The fact that the series \eqref{eq:eigenfunctionExp}
converges in $L^{2} (\nu_{L}\times \nu_{L})$ is of no use here,
because for any $s,t>0$ the joint distribution of $(\gamma (s),\gamma
(t))$ is \emph{singular} relative to $\nu_{L}\times \nu_{L}$.  If the
kernel $H_{\delta} (u,v)$ were continuous and positive semi-definite,
then Mercer's theorem (\cite{courant-hilbert}, ch.~3) would imply
\emph{pointwise} --- in fact, uniform --- convergence of the series;
unfortunately,  neither $H_{\delta}$ nor its smoothed
version $K_{\delta}$ is positive semi-definite.

The way around these difficulties, it seems, is not
to use the eigenfunction expansion, but instead to
approximate, in $C^{m}-$norm for suitable $m$, the smoothed kernel
$K_{\delta} (x,y)$ by \emph{elementary kernels}, that is, finite sums of the form
\begin{equation}\label{eq:elem}
	h (x,y):=\sum_{i=1}^{J}\sum_{j=1}^{J} a_{i,j}\varphi_{i} (x)\varphi_{j} (y),
\end{equation}
where the functions $\varphi_{j}$ are $C^{m}$, with mean zero relative
to $\nu_{L}$, and the matrix $(a_{i,j})$ is symmetric. A repetition of
the argument given above shows that for any such kernel $h$,
\[
	T^{-1}\iint_{[0,T]^{2}}  h(\gamma (s),\gamma
	(t)) \,ds dt \stackrel{\mathcal{D}}{\longrightarrow}  G
\]
where $G$ is a (possibly degenerate) quadratic form in independent
unit Gaussian random variables. Thus, to prove
Theorem~\ref{theorem:random} it will suffice to show that the error
incurred in approximating $K_{\delta}$ by $h$ in the integral
\[
	T^{-1}\iint_{[0,T]^{2}}  K_{\delta}(\gamma (s),\gamma
	(t)) \,ds dt
\]
is small uniformly for $T>1$.  In the remainder of this section we shall show that
such estimates reduce to a problem of mixing for the geodesic flow.

\subsection{Multiple Mixing Rates for the Geodesic
Flow}\label{ssec:elementary}

The proof of Theorem~\ref{theorem:random} will rely on an extension of
\textsc{Dolgopyat}'s theorem \cite{dolgopyat} on exponential mixing
rates for the geodesic flow.  For the sake of simplicity, since our
only interest is in the case of the Liouville measure $\nu_{L}$, we
will state the extension only for this measure. Denote by $\zz{E}$
expectation relative to $\nu_{L}$, and by $\gamma_{t}=\gamma (t)$ the
geodesic ray whose initial direction vector $\gamma_{0}$ is randomly
chosen according to $\nu_{L}$.

\begin{theorem}\label{theorem:mixing-elementary}
For each $K=2,3,\dotsc$, there exist constants $C=C_{K}<\infty$,
$A=A_{K}>0$, and $m=m_{K}\in \zz{N}$ such that for any mean-zero,
real-valued functions $F_{1},F_{2},\dotsc ,F_{K}$ of class $C^{m}$ on
$SM$ and all $-\infty =t_{0}<t_{1}<t_{2}<\dotsb <t_{K}<t_{K+1}=\infty$,
\begin{equation}\label{eq:mixing-elementary}
	\bigg|\zz{E} \prod_{j=1}^{K} F_{j} (\gamma (t_{j}))\bigg|
	\leq C \left( \min_{1\leq j \leq K} \max (e^{-A (t_{j+1}-t_{j})}, e^{-A
	(t_{j}-t_{j-1})})\right)
	\prod_{j=1}^{K} \xnorm{F_{j}}_{m}.
\end{equation}
Here $\xnorm{F}_{m}$ denotes the $C^{m}$ norm of the function
$F$, that is,
\begin{equation}\label{eq:mNorm}
	\xnorm{F}_{m} := \min_{k\leq m} \sup_{x\in SM}
	|{D_{i_{1}}D_{i_{2}}\dotsb D_{i_{k}}F (x)}|
\end{equation}
where the maximum is over all mixed partial derivatives of order $\leq
m$ in the directions of unit tangent vectors to $SM$.
\end{theorem}

Dolgopyat \cite{dolgopyat} proves in the case $K=2$ (for functions of
differentiability class $C^{7}$) that for $t>0$,
\begin{equation}\label{eq:dolg}
	\bigg| \zz{E} F_{0} (\gamma (0)) F_{1} (\gamma (t)) \bigg|
	\leq  C e^{-At} \xnorm{F_{0}} \xnorm{F_{1}}.
\end{equation}
The inequality \eqref{eq:mixing-elementary} is a natural extension of
this, and can be proved in a similar fashion (using in addition the
induction strategy in \cite{dolgopyat-2}).   Because the arguments are
so similar, we omit the details. (In fact, the argument below -- see
Lemma~\ref{lemma:chebyshev} -- require only the case $K=4$.)

\subsection{Approximation by Elementary
Kernels}\label{ssec:approximation} To make use of
Theorem~\ref{theorem:mixing-elementary} we will approximate the
{centered} kernels $K^{*}_{\delta} (x,y):=K_{\delta}
(x,y)-\kappa_{\varphi}$ by \emph{elementary kernels}, that is, kernels
of the form \eqref{eq:elem}:
\[
	h (x,y):=\sum_{i=1}^{J}\sum_{j=1}^{J} a_{i,j}\varphi_{i} (x)\varphi_{j} (y),
\]
where the functions $\varphi_{j}$ are $C^{m}$, with mean zero relative
to $\nu_{L}$, and the matrix $(a_{i,j})$ is symmetric.  Existence of
such approximations (in particular, the requirement that the functions
$\varphi_{j}$ have mean zero) depends {crucially} on the fact that the
constant function $1$ is an eigenfunction of $K_{\delta}$ (Property
(P1) above): this guarantees that the kernel $K^{*}_{\delta} (x,y)$ is
\emph{centered}, that is,
\begin{equation}\label{eq:centered}
	\int K^{*}_{\delta} (x,y) \, \nu_{L} (dy)=
	\int K (x,y) \, \nu_{L} (dy) -\kappa_{\varphi} =0 
	\quad \text{for every} \;\; x\in SM.
\end{equation}

\begin{lemma}\label{lemma:approx}
Let $g (x,y)$ be a centered, symmetric kernel of class $C^{m}$, where
$m\geq 0$.  Then for every $\varepsilon >0$ there is an elementary
kernel $h (x,y)$ such that
\begin{equation}\label{eq:CmApprox}
	\xnorm{g-h}_{m}<\varepsilon .
\end{equation}
\end{lemma}

\begin{proof}
This is completely standard except for the requirement that the
component functions $\varphi_{j}$ in the approximating kernel $h$ be
mean-zero. Suppose that $g$ is well-approximated by a kernel $h$ of
the form \eqref{eq:elem} in which the component functions
$\varphi_{j}$ are \emph{not} mean-zero. Denote by $\bar{\varphi}_{j}$
the mean of $\varphi_{j}$ relative to $\nu_{L}$. Since $g (x,y)$
integrates to $0$ against $\nu_{L} (dy)$ for every $x$,  the
inequality \eqref{eq:CmApprox} implies that 
\[
	\bigg| \sum_{i=1}^{J}\sum_{j=1}^{J} a_{i,j}\bar{\varphi}_{j}\varphi_{i} (x)
	\bigg|
 	<\varepsilon ,\quad 
	\text{whence} \quad
	\bigg| \sum_{i=1}^{J}\sum_{j=1}^{J} a_{i,j}\bar{\varphi}_{i}\bar{\varphi}_{j}
	\bigg|
 	<\varepsilon.
\]
Consequently, $g (x,y)$ is also 
well-approximated by the kernel
\[
	h^{*} (x,y):=\sum_{i=1}^{J}\sum_{j=1}^{J} a_{i,j}\varphi^{*}_{i}
	(x)\varphi^{*}_{j} (y)
	\quad \text{where} \quad 
	\varphi^{*}_{j} (x)  = \varphi_{j} (x)-\bar{\varphi}_{j}.
\]
\end{proof}

We will need a quantitative version of the approximation
\eqref{eq:CmApprox}, in which the $C^{m}-$norms of the terms
$a_{i,j}\varphi_{j} (x)\varphi_{i} (y)$ are controlled by the size of
the kernel $g (x,y)$. This can be done, but at the cost of a more
stringent differentiability requirement on $g (x,y)$.

\begin{lemma}\label{lemma:quantApprox}
For every $m\geq 1$ there exists $C<\infty$ such that the following
holds: For every centered, symmetric kernel $g (x,y)$ of class
$C^{2m}$, the elementary kernel $h(x,y)$ in the approximation
\eqref{eq:CmApprox}  can be chosen so that 
\begin{equation}\label{eq:quantApprox}
	|a_{i,j}|\leq C \xnorm{g}_{2m} (i+j)^{-m} 
	\quad \text{and} \quad \xnorm{\varphi_{j}}_{m}\leq 1.
\end{equation}
\end{lemma}

\begin{proof}
First, use a smooth partition of the identity to localize, then use
Fourier series approximations in the coordinate patches. The
coordinate patches $U_{i}$ in $SM\times SM$ can be  chosen so that they are
nearly isometric to cubes in $\zz{R}^{6}$; using a matching partition
of $1$, we obtain a decomposition
\[
	g = \sum_{k=1}^{K} g_{i}
\]
where each $g_{i}$ is supported by $U_{i}$ and has $C^{2m}-$norm
bounded by $C'\xnorm{g}_{2m}$, with a constant $C'$ independent
of the kernel $g$. Each $g_{i}$ may now be viewed as a
$C^{2m}-$function on a cube in $\zz{R}^{6}$, and as such can be
expanded in a Fourier sine series. (This may have a nonzero constant
term, but  this will wash out later, since the original kernel $g$ is
centered.) Now $\sin k x$, as a function on $[0,2\pi]$, has
$L^{2}-$norm $1/\sqrt{2}$, but has $C^{m}-$norm $k^{m}$. However,
because $g_{i}$ is of class $C^{2m}$, its inner product with any
product of sines can be integrated by parts up to  $2m$ times, so  the
Fourier coefficient of any sine product containing a factor $\sin kx$
will be bounded in magnitude by $C''/k^{2m}$. Thus, the 
resulting series approximation will satisfy \eqref{eq:quantApprox} for suitable
$C<\infty$. Finally, the resulting approximations to $g$ can be
modified so that the component functions $\varphi_{k}$ are mean zero,
by the same argument as in the proof of Lemma~\ref{lemma:approx}.
\end{proof}
 
\subsection{$L^{2}$ Bounds via Approximation}\label{ssec:L2bounds}
Lemma~\ref{lemma:approx} asserts that
every centered, symmetric kernel $g (x,y)$ of class $C^{m}$ can be
arbitrarily well-approximated, in the $C^{m}$ norm, by
\emph{elementary} kernels $h (x,y)$. This leaves the problem of
determining how much of an error might be incurred in replacing $g$ by
$h$ in the integral
\begin{equation}\label{eq:targetIntegral}
	T^{-1}\iint_{[0,T]^{2}} g (\gamma (s),\gamma (t)) \,ds dt.
\end{equation}
It is obvious that if $\xnorm{g-h}_{m}<\varepsilon$ then the error
cannot exceed $\varepsilon T$. But this isn't good enough for
our purposes: we need the error to be of size $O (1)$
for large $T$. This is where the multiple-mixing rate provided by
Theorem~\ref{theorem:mixing-elementary} comes in. 

\begin{lemma}\label{lemma:chebyshev}
For $m\in \zz{N}$  sufficiently large there  exist constants
$C=C_{m}<\infty$ such that if $g(x,y)$ is a centered, symmetric
kernel of class $C^{2m}$ then for all $T\geq 1$,
\begin{equation}\label{eq:chebyshev}
	\zz{E} \left(\iint_{[0,T]^{2}}  g(\gamma (s),\gamma
	(t)) \,ds dt\right)^{2}  \leq C T^{2} \xnorm{g}_{2m}^{2}.
\end{equation}
\end{lemma}

\begin{proof}
Let $g (x,y)$ be a centered, symmetric kernel of class $C^{2m}$. By
Lemma~\ref{lemma:approx}, for each $T\geq 1$ there exist an elementary
kernel $h=h_{T}$ such that $g-h$ has $C^{m}-$norm -- and therefore
also sup norm -- smaller than $\xnorm{g}_{2m}/T^{3}$. Furthermore, by
Lemma~\ref{lemma:quantApprox}, the kernel $h=h_{T}$ can be chosen so
that it is of the form \eqref{eq:elem}:
\[
	h (x,y) =\sum_{i}\sum_{j} a_{i,j} \varphi_{i} (x)\varphi_{j} (y),
\]
 where the functions $\varphi_{i}$ are mean-zero with $C^{m}-$norms
bounded by $1$, and the coefficients $a_{i,j}$ satisfy
\[
	|a_{i,j}|\leq C' \xnorm{g}_{2m}/ (i+j)^{m}.
\]
Since the sup norm of $g-h$ is smaller than $\xnorm{g}_{2m}/T^{3}$, it
follows that
\[
	\bigg|\iint_{[0,T]^{2}}  g(\gamma (s),\gamma
	(t)) \,ds dt -
	\iint_{[0,T]^{2}}  h(\gamma (s),\gamma
	(t)) \,ds dt \bigg| \leq C''T^{-1}\xnorm{g}_{2m}.
\]
Consequently, it suffices to establish the inequality
\eqref{eq:chebyshev} with $g$ replaced by $h=h_{T}$. But because 
$h$ is elementary,
\begin{multline*}
	\zz{E}\left(\iint_{[0,T]^{2}}  h(\gamma (s),\gamma
	(t)) \,ds dt \right)^{2}\\
	 =4!\sum_{i_{1}}\sum_{i_{2}} \sum_{i_{3}}
	\sum_{i_{4 }}
	a_{i_{1},i_{2}}a_{i_{3},i_{4}}\iiiint_{0<s_{1}<s_{2}<s_{3}<s_{4}<T}\zz{E}
	\prod_{j=1}^{4}\varphi_{i_{j}} (\gamma (s_{j}))\,ds_{1}ds_{2}ds_{3}ds_{4}
\end{multline*}
Now Theorem~\ref{theorem:mixing-elementary} applies: in particular,
since each of the functions $\varphi_{j}$ has mean $0$ relative to
Liouville measure, and since each has $C^{m}-$norm no larger than $1$,
Theorem~\ref{theorem:mixing-elementary} 
implies that the inner expectation is bounded in magnitude by 
\[
	C''' \min (\exp \{-A ( s_{2}-s_{1})\},\exp \{-A ( s_{4}-s_{3})\}).
\]
Thus, the quadruple integral is bounded by a constant multiple of
$T^{2}$.
\end{proof}

\subsection{Weak Convergence for Centered Kernels}\label{ssec:weakConv} 
The approximation results of section \ref{ssec:approximation}
together with the $L^{2}-$bound provided by
Lemma~\ref{lemma:chebyshev} together imply that the normalized
integral \eqref{eq:targetIntegral} has, for large $T$, a distribution
close to that of a Gaussian quadratic form:

\begin{theorem}\label{theorem:approxConv}
Let $g (x,y)$ be a centered, symmetric kernel of class $C^{\infty}$.
Then as $T \rightarrow \infty$,
\begin{equation}\label{eq:approxConv}
		T^{-1}\iint_{[0,T]^{2}} g (\gamma (s),\gamma (t)) \,ds
		dt \stackrel{\mathcal{D}}{\longrightarrow} F ,
\end{equation}
where $F$ is a probability distribution in the weak closure of the set of
Gaussian quadratic forms. 
\end{theorem}

\begin{proof}
By Lemma \ref{lemma:approx}, for each $m\geq 1$ and any $\varepsilon
>0$ there is an elementary function $h (x,y)$ of class $C^{m}$ such
that the difference $r:=g-h$ has $C^{m}-$norm less than
$\varepsilon$. Since $h$ is elementary, it is centered and symmetric;
hence, so is the difference $r$. Therefore, by
Lemma~\ref{lemma:chebyshev}, if $m$ is sufficiently large then for all $T>1$,
\[
	\zz{E} \left(T^{-1}\iint_{[0,T]^{2}}  r(\gamma (s),\gamma
	(t)) \,ds dt\right)^{2}  \leq C \varepsilon 
\]
where $C=C_{m}<\infty$ is a constant depending only on the differentiability
class $C^{m}$. On the other hand, since $h$ is elementary, it has the
form
\[
		h (x,y)=\sum_{i=1}^{J}\sum_{j=1}^{J} a_{i,j}\varphi_{i} (x)\varphi_{j} (y),
\]
that is, it is a finite, symmetric quadratic form in the functions
$\varphi_{j}$. Consequently, by the same argument as in section
\ref{ssec:heuristics} above, as $T \rightarrow \infty$,
\[
	T^{-1}\iint_{[0,T]^{2}}  h(\gamma (s),\gamma
	(t)) \,ds dt \stackrel{\mathcal{D}}{\longrightarrow}  G
\]
where $G$ is a (possibly degenerate) quadratic form in independent
unit Gaussian random variables. Since $\varepsilon >0$ can be chosen
arbitrarily small, it follows that for large $T$ the random variable 
\[
	T^{-1}\iint_{[0,T]^{2}}  g(\gamma (s),\gamma
	(t)) \,ds dt 
\]
is arbitrarily close in $L^{2}$ to a random variable with a
distribution close to that of a Gaussian quadratic form. The theorem
follows. 
\end{proof}

\begin{remark}\label{remark:negatives}
Unfortunately, it seems, this argument gives no information about the
limit distribution $F$ other than the fact that it is a weak limit of
distributions of Gaussian quadratic forms. The principal difficulty is
in the use of the central limit theorem for the geodesic flow: Even if
one knows, say, that the functions $\varphi_{i}$ in the elementary
kernel $h$ are uncorrelated under $\nu_{L}$, it does not follow that
the random variables
\[
	\frac{1}{\sqrt{T}} \int_{0}^{T} \varphi_{i} (\gamma (s)) \,ds
\]
are  uncorrelated, nor that their Gaussian limits will be
uncorrelated. Thus, it seems that it is impossible to relate the
coefficients in the limiting quadratic form to the coefficients
$a_{i,j}$ in the expansion of the elementary approximations.
\end{remark}

\subsection{Proof of
Theorem~\ref{theorem:random}}\label{ssec:theProof}
Theorem~\ref{theorem:approxConv} applies to any centered, symmetric
kernel of class $C^{\infty}$, and so in particular to the kernels
$K_{\delta}$ defined by \eqref{eq:kernelK}. Consequently, to complete
the proof of Theorem~\ref{theorem:random} it suffices to show that the
error in the inequalities \eqref{eq:bounds} is small (in $L^{1}
(\nu_{L})$) compared to $t$, for large $t$. Recall
(Remark~\ref{remark:errors}) that the error in \eqref{eq:bounds} is
bounded by
\begin{align}\label{eq:errorBounds}
	N_{\varphi}(\gamma [-\delta ,t+\delta])-N_{\varphi}(\gamma
	 [\delta,t-\delta])&= M_{\varphi} (\gamma [-\delta ,\delta],\gamma [-\delta
	 	       ,t+\delta])  \\
\notag &+M_{\varphi} (\gamma [-\delta ,t+\delta],\gamma [t-\delta
	 	       ,t+\delta]) 
\end{align}
where $M_{\varphi} (\alpha ,\beta)$ denotes the weighted sum of the
transversal intersections between the geodesic segments $\alpha$ and
$\beta$, that is,
\[
	M_{\varphi} (\alpha ,\beta)=\sum \varphi (\theta_{i})
\]
where the sum is over all intersections $i$ between $\alpha$ and
$\beta$, and $\theta_{i}$ is the angle of the $i$th
intersection. Hence, since the Liouville measure $\nu_{L}$ is
invariant and reversible under the geodesic flow, the \emph{expected}
error in \eqref{eq:bounds} is bounded by
\[
	2 \zz{E}M_{\varphi} (\gamma
	[0,2\delta],\gamma [2\delta, t+4\delta]).
\]
Thus, it suffices to show that this expectation is small compared to
$t$ when $\delta$ is small. This is implied by the following lemma,
which completes the proof of Theorem~\ref{theorem:random}.

\begin{lemma}\label{lemma:errorBoundsFinal}
\begin{equation}\label{eq:errorBoundsFinal}
	\lim_{\delta  \rightarrow 0} \sup_{t\geq 1} t^{-1}\zz{E}M_{\varphi}
		      (\gamma [0,\delta],\gamma [0,t])= 0.
\end{equation}
\end{lemma}

\begin{proof}
Recall that geodesic segments of length $\leq \varrho$ can intersect
at most once, provided $\varrho>0$ is sufficiently
small. Hence, if $\delta <\varrho$ then $M_{\varphi} (\gamma
[0,\delta],\gamma [0,t])$ cannot be larger than
$\xnorm{\varphi}_{\infty}t/\varrho$. Thus, the expectation in
\eqref{eq:errorBounds} is $O (t)$. The problem is to prove that the
implied constant shrinks to $0$ as $\delta \rightarrow 0$. For this we
will use the SLLN \eqref{eq:llnPhi} (proved in
section~\ref{sec:slln}) and an averaging trick. The averaging trick
is this: since $M_{\varphi}$
is additive,
\begin{equation}\label{eq:brackets}
	\delta^{-1}\int_{s=0}^{t} M_{\varphi} (\gamma
	[s,s+\delta],\gamma [0,t]) \,ds \leq 2N_{\varphi} (\gamma
	[0,t]) \leq \delta^{-1}\int_{s=-\delta }^{t+\delta } M_{\varphi} (\gamma
	[s,s+\delta],\gamma [0,t]) \,ds.
\end{equation}
This implies that $N_{\varphi} (\gamma [0,t])$ is bounded above by
$\xnorm{\varphi}_{\infty} t (t+2\delta)/ (\delta \varrho)$, and so in
particular the random variables $N_{\varphi} (\gamma [0,t])/t^{2}$ are
uniformly bounded. Therefore,
the SLLN implies convergence of expectations: 
\begin{equation}\label{eq:expectationConvergence}
	\lim_{t \rightarrow \infty} 2\zz{E}N_{\varphi} (\gamma
	[0,t])/t^{2}=\kappa_{\varphi}. 
\end{equation}
Next, by additivity of $M_{\varphi}$ and reversibility of the geodesic
flow relative to the Liouville measure,
\begin{align*}
	\zz{E} M_{\varphi} (\gamma
	[s,s+\delta],\gamma [0,t])
	&=
	\zz{E} M_{\varphi} (\gamma [0,\delta],\gamma [0,t-s-\delta])\\
	& +
	\zz{E} M_{\varphi} (\gamma [0,\delta],\gamma [0,s+\delta]).
\end{align*}
Substituting this in \eqref{eq:brackets} yields
\[
	2\delta^{-1}\int_{s=\delta }^{t-\delta} \zz{E} M_{\varphi} (\gamma [0,\delta],\gamma
	[0,s]) \,ds \leq 2\zz{EN_{\varphi}} (\gamma [0,t]).
\]
Since $M_{\varphi} (\gamma [0,\delta],\gamma [0,s])$ is nondecreasing
in $s$, it follows that 
\[
	\zz{E} M_{\varphi} (\gamma [0,\delta],\gamma [0,t/2])/t
	\leq 2\delta \zz{E}N_{\varphi} (\gamma [0,t])/t (t-\delta ).
\]
The desired result \eqref{eq:errorBoundsFinal} now follows from the
convergence of expectations \eqref{eq:expectationConvergence}.
\end{proof}

\begin{remark}\label{remark:sharpErrors}
A more sophisticated argument, using the central limit theorem for the
geodesic flow, shows that the errors in the inequality
\eqref{eq:bounds} are actually of order $\sqrt{t}$.
\end{remark}

\section{Central Limit Theorem for $N_{\varphi ;f} (t)$}\label{sec:local}

\subsection{Localized Intersection Kernels}\label{ssec:local-kernels}
In proving Theorem~\ref{theorem:local} there is no loss
of generality in considering only \emph{nonnegative} functions $f$, so
we shall assume throughout that $f:M \rightarrow \zz{R}$ is a
{nonnegative}, $C^{\infty}$ function and that the smoothing function
$\varphi$ satisfies the hypotheses of
Theorem~\ref{theorem:local}. When convenient, we will view $f$ as a
function on $SM$: that is, for any $u= (x,\theta)\in SM$, set $f (u)=f
(x)$.  Recall that the $f-$localized self-intersection counts
$N_{\varphi ;f} (t)$ are obtained by summing $f (x_{i})\varphi
(\theta_{i})$, where $x_{i}$ and $\theta_{i}$ are the locations and
angles of the self-intersections, and the sum is over all
self-intersections $i$ of the geodesic segment $\gamma [0,t]$. As for
the global self-intersection counts, the localized counts can be
represented by \emph{intersection kernels}. 
As is the case for the global intersection
kernels, the obvious local kernels  are not continuous, despite
the fact that  the functions $\varphi$ and $f$ are both
$C^{\infty}$. Consequently, we define smooth kernels $k_{\delta}$ as
in Section~\ref{ssec:smoothing}:
\begin{equation}\label{eq:local-smoothed}
	k_{\delta} (u,v)=\sum_{i}\delta^{-2}p (s_{i}/\delta)p
	(t_{i}/\delta) f (x_{i}) \varphi (\theta_{i})
\end{equation}
where $i, s_{i},t_{i},$ and $p$ are as in the definition
\eqref{eq:kernelK}. By the same argument as in
Lemma~\ref{lemma:well-def}, if $\delta>0$ is sufficiently small then
the sum \eqref{eq:local-smoothed} contains at most one nonzero term,
and so $k_{\delta}$ extends to a symmetric, $C^{\infty}$ function on
$SM\times SM$. In addition, by the same argument as in
Lemma~\ref{lemma:siFormula},
\begin{equation}\label{eq:local-bounds}
	N^{*}_{\varphi ;f} (\gamma [\delta ,t-\delta ])
	\leq N_{\varphi ;f}(\gamma [\delta ,t-\delta ])
	\leq 	N^{*}_{\varphi ;f} (\gamma [-\delta ,t+\delta ])
\end{equation}
where
\begin{equation}\label{eq:Nstar}
	N^{*}_{\varphi ;f} (\gamma [a,b]):=\frac{1}{2}
	\iint_{[a,b]^{2}} k_{\delta} (\gamma (r),\gamma (s)) \,dr ds.
\end{equation}
Furthermore, by the
same argument as in Section~\ref{ssec:theProof}, the errors in the
inequalities \eqref{eq:local-bounds} are of order $O (t)$. Since the
limit relation  \eqref{eq:clt-local}  involves fluctuations  of  order
$t^{3/2}$, it follows that the errors in \eqref{eq:local-bounds} can
be ignored in proving Theorem~\ref{theorem:local}. Thus, it
suffices to prove, for some small $\delta >0$, that as $t \rightarrow \infty$,
\begin{equation}\label{eq:central limit theorem-local-smoothed}
		\frac{N^{*}_{\varphi ;f} (\gamma [0,t])-A_{\varphi;f}
t^{2}}{\sigma t^{3/2}} \Longrightarrow \Phi .
\end{equation}

\begin{corollary}\label{corollary:Adelta} (SLLN)
Define $2A_{\delta}=\kappa_{\varphi} \zz{E} f_{\delta}$ where the
expectation is with respect to the Liouville measure. Then for all
sufficiently small $\delta >0$,
\begin{equation}\label{eq:Adelta}
	A_{\delta}=A_{\varphi ;f} \quad \text{and} \quad 
	\lim_{t \rightarrow \infty} \frac{N^{*}_{\varphi ;f}}{t^{2}}
	=	\lim_{t \rightarrow \infty} \frac{N_{\varphi
	;f}}{t^{2}} =A_{\varphi ;f} \;\;a.s.
\end{equation}
\end{corollary}

\begin{proof}
Proposition~\ref{proposition:ergodic} implies that $N^{*}_{\varphi ;f}
(t)/t^{2}$ converges almost surely to $A_{\delta}$. But the
inequalities \eqref{eq:local-bounds} and the fact that the errors are
of order $O (t)$ imply that the limit constant does not depend on $\delta$.
\end{proof}

\subsection{Lead Eigenfunction}\label{ssec:leadEig}
The primary difference between the global intersection kernel
$K_{\delta}$ and the local kernel  $k_{\delta}$
is that the constant function $1$ is not, in general, an eigenfunction
of the local kernel. To see this, define
\begin{equation}\label{eq:Fdelta}
	f_{\delta} (u) := \frac{1}{\kappa_{\varphi}}\int_{SM}  k_{\delta} (u,v) \,\nu_{L} (dv).
\end{equation}

\begin{lemma}\label{lemma:FdeltaF}
\[
	\lim_{\delta  \rightarrow 0} \xnorm{f_{\delta}-f}_{\infty}=0.
\]
\end{lemma}

\begin{proof}
Because $k_{\delta}$ is non-zero only for pairs $u,v$ at distance $<2\delta$,
the value of $f$ in \eqref{eq:local-smoothed} will be close to $f (u)$
for all small $\delta$, uniformly for $u\in SM$. (Here we are viewing
$f$ as a function on $SM$.) Hence,
\[
	|f (u) K_{\delta} (u,v) - k_{\delta} (u,v)| \leq
	\max_{d (a,b)\leq \delta} |f (a)-f (b)|.
\]
Since $K_{\delta }1=\kappa_{\varphi}$, by (P1) of
section~\ref{ssec:smoothing}, the result follows from the continuity
of $f$.
\end{proof}

\subsection{Cohomology and the Central Limit Theorem}\label{ssec:coho}
By Lemma~\ref{lemma:FdeltaF}, if $f$ is not cohomologous to a
constant, then neither is $f_{\delta}$ provided $\delta >0$ is
sufficiently small. This follows trivially from the definition:

\begin{definition}\label{definition:cohomology}
A $C^{\infty}$ function $g:SM \rightarrow \zz{R}$ (or $\zz{C}$) is said
to be a \emph{coboundary} relative to the geodesic flow if it
integrates to zero along every closed geodesic; similarly, $g$ is said
to be \emph{cohomologous} to a constant $a$  if $g-a$ is a
coboundary. 
\end{definition}

 It is not so easy to find nonconstant functions that are cohomologous
to constants, but it is quite easy to construct a function $g$ that is
not cohomologous to a constant: Take two closed geodesics $\alpha$ and
$\beta$ that do not intersect\footnote{That there are non-intersecting
closed geodesics on any compact surface of \emph{constant} negative
curvature follows from the standard representation of such a surface
as a geodesic polygon in the hyperbolic plane with sides identified,
because non-crossing sides will be non-intersecting closed geodesics
after the identifications are made. It then follows that there are
non-intersecting closed geodesics on any compact surface of variable
negative curvature, because all variable-negative-curvature metrics on
a compact surface can be obtained by smooth deformation of the
constant curvature metric, and such deformations preserve transversal
intersections of closed geodesics. For more detail, see, for instance,
\cite{lalley:rigidity}.} on $M$, and let $g:M \rightarrow \zz{R}$ be
any $C^{\infty}$, nonnegative function that is identically $1$ along
$\alpha$ but vanishes in a neighborhood of $\beta$. In fact, the
existence of non-intersecting closed geodesics yields the existence of
a large class of functions that are not cohomologous to constants:

\begin{proposition}\label{proposition:plethora}
Let $\varepsilon >0$ be the distance in $M$ between two
non-intersecting closed geodesics $\alpha$ and $\beta$. Then no
$C^{\infty}$, nonnegative, function $g:M \rightarrow \zz{R}$ that is
not identically zero and whose support has diameter less than
$\varepsilon$ is cohomologous to a constant.
\end{proposition}

\begin{proof}
By hypothesis, $g$ vanishes on at least one of the geodesics
$\alpha,\beta$. Because closed geodesics are dense in $SM$, their
projections are dense in $M$. Thus, since $g$ is not identically $0$,
there is a closed geodesic $\xi$ on which the average value of $g$ is
positive.
\end{proof}

The importance of the concept of cohomology to us is its relation to
the central limit theorem for the geodesic flow: If $g$ is \emph{not}
cohomologous to a constant, then there exists\footnote{See
\cite{ratner:clt} for the assertion that $\sigma_{f}>0$ if and only if $f$
is not cohomologous to a constant. The definition of cohomology used
there is different from ours, but it is easily shown that the two
definitions are equivalent.}
$\sigma =\sigma_{g}>0$ such that if $\gamma (0)$ is chosen at random
from the Liouville measure, then
\begin{equation}\label{eq:clt-coho}
	\frac{1}{\sigma \sqrt{t}}\left\{ \int_{0}^{t} g (\gamma (s))
	\,ds - \int_{SM} g \,d\nu_{L}\right\} \Longrightarrow \Phi .
\end{equation}

\subsection{Proof of Theorem~\ref{theorem:local}}\label{ssec:clt}
Rewrite the relation \eqref{eq:Nstar} as
\begin{equation}\label{eq:NstarRewrite}
	2N^{*}_{\varphi ;f} (\gamma [0,t]) =
		       2\kappa_{\varphi }\iint_{[0,t]^{2}} f_{\delta} (\gamma
		       (r)) \,dr ds - 2A_{\delta}t^{2}
		       +\iint_{[0,t]^{2}} \tilde{k}_{\delta} (\gamma
		       (r),\gamma (s)) \,dr ds 
\end{equation}
where 
\[
	\tilde{k}_{\delta} (u,v) = k_{\delta} (u,v) -\kappa_{\varphi } (f_{\delta}
	(u)+f_{\delta} (v)) +A_{\delta}.
\]
Corollary~\ref{corollary:Adelta} implies that $A_{\delta}=A_{\varphi
;f}$ for all small $\delta >0$.
By Proposition~\ref{proposition:plethora}, the function $f$ is not
cohomologous to a constant, and so by Lemma~\ref{lemma:FdeltaF}
neither is $f_{\delta}$, provided $\delta>0$ is sufficiently
small. Consequently, the fluctuations of the first integral in
\eqref{eq:NstarRewrite} are of order $t^{3/2}$, and the central limit
theorem implies that for $\sigma =\sigma_{f,\delta}>0$,
\[
	\frac{1}{\sigma t^{3/2}} 
	\left\{ \kappa_{\varphi }\iint_{[0,t]^{2}} f_{\delta} (\gamma
		       (r)) \,dr ds - A_{\varphi ;f}t^{2}\right\} 
		     \Longrightarrow \Phi .		      
\] 
On the other hand, since the kernel $\tilde{k}_{\delta}$ is
$C^{\infty}$, symmetric, and \emph{centered}, the fluctuations of the
second integral in \eqref{eq:NstarRewrite} are of order $t$, by
Theorem~\ref{theorem:approxConv}. Theorem~\ref{theorem:local}
follows. 
\qed

\bigskip \noindent 
\textbf{Acknowledgment.} Thanks to \textsc{Jayadev Athreya} for a
helpful discussion.
 
\bibliographystyle{plain}
\bibliography{mainbib}

%
%

\end{document}